\documentclass[12pt]{amsart}

\usepackage[dvipsnames]{xcolor}
\usepackage{pgf,tikz}
\usetikzlibrary{arrows}
\usepackage{pifont}

\usepackage{amsmath}
\usepackage{amssymb}
\usepackage{latexsym}
\usepackage{enumerate}
\usepackage{amscd}

\usepackage{tikz}
\usepackage[all]{xy}
\usepackage{graphics}
\usepackage[active]{srcltx}

\newcommand{\HH}{\mathrm{H}}
\newcommand{\HF}{\operatorname{H\underline{\F}}}

\newcommand{\C}{\mathbf{C}}
\newcommand{\F}{\mathbf{F}}

\newcommand{\classB}{B(\C_2, GL_8(\mathbb R))}

\numberwithin{equation}{section}

\newtheorem{theorem}[equation]{Theorem}
\newtheorem{proposition}[equation]{Proposition}
\newtheorem{corollary}[equation]{Corollary}
\newtheorem{lemma}[equation]{Lemma}

\theoremstyle{definition}

\newtheorem{definition}[equation]{Definition}

\newcommand{\wmono}{ \ar@{>->}[r]}
\newcommand{\wmonovert}{ \ar@{>->}[d]}
\newcommand{\cof}{ \ar@{^{(}->}[r]}

\begin{document}

\title{Floyd's manifold is a conjugation space}

\author{Wolfgang Pitsch}
\address{Universitat Aut\`onoma de Barcelona \\ Departament de Matem\`atiques\\
E-08193 Bellaterra, Spain}
\email{pitsch@mat.uab.es}

\author{J\'er\^ome Scherer}
\address{EPFL \\ Institute of Mathematics\\ Station 8, CH-1015 Lausanne, Switzerland}
\email{jerome.scherer@epfl.ch}

\thanks{The authors are partially supported  by FEDER/MEC grant MTM2016-80439-P.}


\newcommand{\W}{\mathcal{W}}
\newcommand{\A}{\mathcal{A}}

\begin{abstract}
We prove that there is an action of the cyclic group $\C_2$ on the $10$-dimensional Floyd manifold which turns it into
a conjugation manifold. The submanifold of fixed points is the $5$-dimensional Floyd manifold, whose cohomology is isomorphic to that
 of the large one, scaled down by dividing the cohomological degree by a factor two.
\end{abstract}
\maketitle

\section{Introduction}\label{sec:intro}

In \cite{MR0334256} Floyd showed, among many other things, that there exist only two non-trivial cobordism classes that contain manifolds
with three cells, and that they lie in dimensions $10$ and $5$. He also produced representatives of these classes, 
which we name here $F_{10}$ and $F_5$.
His computations show that the mod $2$ cohomology of the larger one has generators $1, u_4, u_6, u_{10}$ in degrees 
$0$, $4$, $6$, and $10$ respectively, whereas the smaller one looks like a divided by $2$ version with generators $1, v_2, v_3, v_5$
in degrees $0$, $2$, $3$, and $5$ respectively. Moreover, the product of the two
middle generators give the top one, so that the abstract ``halving isomorphism'' between $\HH^*(F_{10}; \mathbb F_2)$ and
$\HH^*(F_{5}; \mathbb F_2)$ is actually compatible with the cup product. More is true since Floyd also identified the action of the Steenrod
algebra and computed that $Sq^2u_4 = u_6$ and $Sq^4 u_6 = u_{10}$, an analogous result being true for $F_5$.
Hence the halving isomorphism is also compatible with the action of the Steenrod algebra.

From the point of view of the theory of conjugation spaces introduced by Hausmann, Holm, and Puppe in \cite{MR2171799}, and 
further developments on the action of the Steenrod algebra by Franz and Puppe, \cite{MR2198191}, this remarkable behavior looks
like an algebraic shadow of the fact that the larger manifold $F_{10}$ could be a conjugation manifold.

We give here an affirmative answer to this conjecture, after several attempts.
This problem appears as an open question in our previous article on the subject, \cite{zoo}, where
this was only proved for the suspension spectrum in the stable homotopy category.  
The main problem is to endow $F_{10}$ with an action by the group of two elements $\C_2$ in such a way that it becomes a conjugation space 
with fixed points~$F_5$.
\medskip

\noindent {\bf Theorem~\ref{thmF10isconjugation}.}
{\it
The Floyd manifold $F_{10}$ is a smooth conjugation manifold with fixed points $F_5$.
}

\medskip

The strategy is to analyze carefully Floyd's surgery construction and perform it equivariantly. For this first step we use L\"uck and Uribe's
theory of equivariant principal bundles, \cite{MR3331607}. This yields an action of $\C_2$ on $F_{10}$. The second and final step consists
then in verifying the ``cohomological purity'' in the genuine equivariant stable sense as explained in our joint work with Ricka, \cite{PRS}. A
key ingredient here is Clover May's splitting result, \cite{Clover}.

\medskip

{\bf Acknowledgements.} We would like to thank Nick Kuhn for posting a question and follow-up on the algebraic topology discussion list
in 2008 about the Floyd manifolds, which appear in Floyd's beautiful article \cite{MR0334256}. We would also like to thank
Nicolas Ricka, Jie Wu, Jean-Claude Hausmann, and Ian Hambleton for enlightening discussions on this subject.

\section{Construction of Floyd's manifolds}\label{sec:Floydsmnflds}

We will go through a  very careful analysis of Floyd's construction  \cite[Section~3]{MR0334256} to ascertain how it inherits a natural $\C_2$ action. 
Let $S^3 \subset \mathbb{C} \oplus \mathbb{C}$ denote the unit sphere in the standard metric. It inherits two intertwined actions: one on the right 
by multiplication by elements in the unit circle $S^1 \subset \mathbb{C}$ and the other on the left by $\C_2$ is induced by conjugation on each factor. 
If $\theta \in S^1$ and $\tau$ denotes the non-trivial element in $\C_2$, then :
  \begin{equation}\label{eqn:S3action}
    \tau(x\theta) = \tau(x) \overline{\theta} \ \ \hbox{\rm for all} \ x \in S^3
    \tag{$\star$}
  \end{equation}
%
As a $\C_2$-representation sphere $S^3 = S^{2 \sigma+1} = S^{2\rho-1}$ in the notation of \cite{PRS}, where $\sigma$ stands for 
the sign representation and $\rho$ for the regular representation.  

Notice that because of (\ref{eqn:S3action}), the action of $S^1$ on $S^3$ restricts to an action of $\C_2 \cong \{\pm 1 \} \subset S^1$ on 
$(S^3)^{\C_2} = S^1$, more precisely it restricts to the antipodal action on the fixed circle. 

Let us define the projective planes $\mathbb{C}P ^2$ and $\mathbb{H}P^2$ as the set of \emph{right} complex lines in $\mathbb{C}^3$ 
(resp. quaternionic lines in $\mathbb{H}^3$). To fix notations, writing $\mathbb{K}$ for $\mathbb{C}$ or $\mathbb{H}$, we have
an element 
\[
[x:y:z] = \{ (x\lambda, y  \lambda, z \lambda) \, | \, \lambda \in \mathbb{K}^\ast \} \in \mathbb{K}P^2
\]
for any $(x,y,x) \in \mathbb{K}^3\setminus \{0,0,0\}$.

Decompose $\mathbb{H}$ as $\mathbb{C} \oplus \mathbb{C}j$. Then $\mathbb{H}$ also has two intertwined actions, 
one by \emph{left} multiplication by elements in the unit circle $S^1$ on the second $\mathbb{C}$-copy, and the other, also on the left, by $\C_2$ and induced by conjugation on each $\mathbb{C}$-copy above. Both actions induce actions on the quaternionic projective plane and these two actions satisfy:
  \begin{equation}\label{eqn:Haction}
    \tau(\theta x) = \overline{\theta} \tau(x) \ \ \hbox{\rm for all} \ x \in \mathbb{H}P^2
    \tag{$\star\star$}
  \end{equation}

To analyze the $\C_2$-fixed points and the inherited action from $S^1$, assume that we have a point $[x:y:z]$ with $z \neq 0$ which is fixed 
(the two other cases $x \neq 0$ and $y \neq 0$  are similar). 
Let us pick the representative of this class of the form $(a,b,1) \in \mathbb{H}^3$. Then
\[
\tau\cdot [a:b:1] = [\overline{a}:\overline{b}:1]
\]
Since $(a,b,1)$ and $(\overline{a},\overline{b},1)$ are in the same right quaternionic line if and only if $a,b \in \mathbb{R} \oplus \mathbb{R}j$, 
this shows that $(\mathbb{H}P^2)^{\C_2} \cong \mathbb{C}P^2$ , where the complex numbers lie ``diagonally'' in $\mathbb H$.
Moreover, since $\theta \cdot 1 = 1$ for $\theta \in S^1$, the $S^1$ action on $\mathbb{H}P^2$ restricts to a $\{ \pm 1 \}$-action on the 
$\C_2$-fixed points $\mathbb{C}P^2$, and this action coincides with the one induced by conjugation on $\mathbb{C} = \mathbb{R} \oplus \mathbb{R}j$
by definition of the $S^1$-action (on the second coordinate only).


Let us consider $\widetilde{F_{10}} = S^3 \times_{S^1} \mathbb{H}P^2$ just like Floyd did in  \cite[Lemma~3.5]{MR0334256}.
Because of the intertwined $S^1$ and $\C_2$-actions on each factor, 
$\widetilde{F_{10}}$ has an induced $\C_2$-action and by the previous computations:
\[
(\widetilde{F}_{10})^{\C_2} = S^1 \times_{\C_2} \mathbb{C}P ^2 = \widetilde{F}_5
\]

Notice that both manifolds $\widetilde{F_{10}}$ and $\widetilde{F_5}$ fit into fibration sequences by projecting on the first factors, 
where the right one is acted upon by $\C_2$ and has as fixed points the left one:
\[
\xymatrix{
\mathbb{C}P^2 \ar[d] \ar@{^{(}->}[r] & \mathbb{H}P^2 \ar[d] \\
S^1 \times_{\C_2}  \mathbb CP^2\ar@{->>}[d] \ar@{^{(}->}[r] &  S^3 \times_{S^1} \mathbb{H}P^2 \ar@{->>}[d]\\
 S^1 \ar@{^{(}->}[r] & S^\rho
}
\]
 The base spaces are respectively $S^1 = \mathbb RP^1 = S^1/\C_2$ and $S^\rho = \mathbb CP^1 = S^3/S^1$. It is immediate to check that the $\C_2$-action is smooth, and we have in fact a smooth fibration. As a direct consequence of~\cite[Theorem 5.3]{MR2171799}:

\begin{proposition}\label{prop:ftildeisconj}
The manifold $\widetilde{F}_{10}$ is a smooth conjugation  manifold with fixed point set $\widetilde{F}_5$.
\hfill{\qed}
\end{proposition}

Since the $S^1$-action  on $\mathbb HP^2$ has a fixed point, which moreover can be chosen to lie on $\widetilde{F_5}$, 
both fibrations have compatible sections:
\[
\xymatrix{
	& S^1 \times_{\C_2} \mathbb{C}P^2\ar@{->>}[dd] \ar@{^{(}->}[rr] & &  S^3 \times_{S^1} \mathbb{H}P^2 \ar@{->>}[dd]\\
	S^1 \ar@{^{(}->}[rr]|\hole \ar@{=}[dr] \ar@{^{(}->}[ur] &  & S^\rho \ar@{=}[dr] \ar@{^{(}->}[ur] & \\
 & 	S^1 \ar@{^{(}->}[rr] &  &S^\rho 
}
\]
where again, the right hand side diagram has a $\C_2$-action with fixed points  left hand side diagram.

To construct his manifolds $F_{10}$ and $F_5$, Floyd proceeds by surgering these split spheres in $\widetilde{F_{10}}$ 
and $\widetilde{F_5}$ and we want to do the same equivariantly on $\widetilde{F_{10}}$ and simultaneously on $\widetilde{F_{5}}$.

\section{The surgery step}\label{sec:surgstep}

We will use the framework developed by L\"uck and Uribe in \cite{MR3331607} to prove that the normal bundle to $S^\rho \hookrightarrow \widetilde{F_{10}}$ is equivariantly trivial. We will work with the  associated frame bundle, as it is a principal $GL_8(\mathbb R)$-bundle
whose triviality is equivalent to that of the normal bundle.

 Because we have split the fibration $\widetilde{F}_{10} \rightarrow S^\rho$ by choosing an $S^1$-fixed point in $\mathbb HP^2$, 
the normal bundle to the splitting $S^\rho \hookrightarrow \widetilde{F}_{10}$ is given by the vertical tangent bundle in the fibration, 
that is the tangent bundle to the fiber $\mathbb{H}P^2$. So, if $T\widetilde{F_{10}}$ is the tangent bundle, 
our model for the normal bundle is $T_V\widetilde{F_{10}}$, the rank $8$ vertical bundle, which fiberwise is the kernel of 
the projection to $S^\rho$. The group $\C_2$ acts by diffeomorphisms on $\widetilde{F_{10}}$, and by equivariance, preserves this bundle: 
for $\gamma \in \C_2$, the tangent map $T\gamma$ acts on the left on $T_V\widetilde{F_{10}}$. The associated frame bundle 
$Fr(T_V \widetilde{F_{10}})$ has as model the bundle one obtains by considering fiberwise, over a point $x \in S^\rho$,
\[
\operatorname{Isom}(\mathbb{R}^8,T_{xV}\widetilde{F_{10}}),
\]
the space of linear isomorphisms out of $\mathbb{R}^8$. The canonical left action of $GL_8(\mathbb{R})$ induces then a right action 
by contravariance on $Fr(T_V \widetilde{F_{10}})$. Precomposition commutes with the postcomposition $\C_2$-action induced by 
the tangent action on $T_{xV}\widetilde{F_{10}}$. Hence the associated frame bundle is an equivariant bundle without intertwine.

We first need to fix a set $\mathcal{R}$ of local isotropy representations, satisfying Condition~(H) of L\"uck-Uribe, \cite[Definition~6.1]{MR3331607}. 
This ensures that the classification problem of the normal bundle is a purely homotopical one by \cite[Theorem~10.1]{MR3331607}, which
means that the equivariant triviality of the bundle will be a consequence of the null-homotopy of the classifying map.

Locally, in a neighborhood of the fixed point that was chosen to split the fibration, the vertical direction looks like 
a quadruple of complex numbers $(z_1,z_2,z_3,z_4)$ on which $\C_2$ acts by complex conjugation as described in the previous section.
Hence the representation of $\C_2$ is given by a homomorphism $\alpha\colon \C_2 \rightarrow GL_8(\mathbb R)$ conjugate
to $I_4 \oplus (-I_4)$.

At a non-fixed point, the isotropy subgroup is per force the trivial group and the representation the trivial one.
We thus set $\mathcal{R}' =\{(1_{\C_2},triv), (\C_2, \alpha) \}$, and close it under conjugation by elements in $GL_8(\mathbb{R})$ to form
a family $\mathcal{R}$. This is, in the terminology of \cite[Definition~3.4]{MR3331607}, the family of representations associated to $S^\rho$. 
Notice that the passage  from pre-families to families is here effortless: the group $\C_2$ is so small that two representations thereon either 
coincide or coincide  only on the trivial element. Because $\C_2$ is finite and $GL_8(\mathbb{R})$ is a quasi-connected Lie group, 
by \cite[Theorem 6.3]{MR3331607}, condition (H) is satisfied for the family $\mathcal{R}$.
These choices define a classifying space $B(\C_2, GL_8(\mathbb{R}), \mathcal{R})$, see \cite[Theorem~11.4]{MR3331607}, which
we write $\classB$ for short. The isomorphism classes of equivariant bundles over $S^\rho$ with isotropy representations in $\mathcal{R}$ are in bijection 
with the set of equivariant homotopy classes into~$\classB$. Hence, to our bundle corresponds a (homotopy class of) $\C_2$-equivariant
map $f\colon S^\rho \rightarrow \classB$. We analyze this map by decomposing $S^\rho$ into the fixed equator $S^1$ and two $2$-cells,
the hemispheres, that are permuted. Our aim is to extend this map to a map $F$ defined on the whole ball $B^\rho$ whose boundary is $S^\rho$. 
We will do this by first extending over the fixed equatorial disk, and then filling-in the two remaining spheres. 
Since we are only interested in the triviality of our bundle, we will make our computations in the reduced setting, i.e. 
we consider $f$ as an element in the set of pointed homotopy classes:
\[
[S^\rho, \classB]^{\C_2}
\]


Our first step is to restrict our equivariant bundle to the fixed circle and show that it is trivial.

\begin{lemma}\label{lem:extendequator}
The classifying map $f$ extends to a map $f'\colon S^\rho \cup D^2 \rightarrow \classB$.
\end{lemma}

\begin{proof}
By Lück-Uribe's results, the restriction of $f$ to $S^1$ amounts to having a non-equivariant bundle on $S^1$ but with a reduced structural group. 
More precisely, since $S^1$ is fixed under the $\C_2$-action, its image lies in the fixed points of the equivariant classifying space.
By \cite[Theorem 13.1]{MR3331607},  these fixed points are
 \begin{eqnarray*}
 \classB^{\C_2} & \simeq & \coprod_{\alpha \in Hom_{\mathbb{R}}(\C_2,GL_8(\mathbb{R}))/GL_8(\mathbb{R})}  BC_{GL_8(\mathbb{R})} (\alpha).
 \end{eqnarray*}
Since up to conjugacy we only have one representation in $\mathcal{R}$, let us fix as a representative the representation that sends $\tau \in \C_2$ to $\left(\begin{matrix} I_4 & 0 \\ 0 & -I_4 \end{matrix}\right)$. 
By direct computation $C_{GL_8(\mathbb{R})} (\alpha) = GL_4(\mathbb{R}) \times GL_4(\mathbb{R})$
and the trivial $ \dashv$ fixed points adjunction yields an isomorphism
\[
[S^1,\classB]^{\C_2} \cong [S^1, \classB^{\C_2}] \cong \pi_1(B( GL_4(\mathbb{R}) \times GL_4(\mathbb{R}))) \cong \mathbb{Z}/2 \oplus  \mathbb{Z}/2.
\]
To analyze our restricted bundle on $S^1$, recall from Floyd's work that the bundle on $S^\rho$ is non-equivariantly two times some bundle. 
So let $g: S^2 \rightarrow BGL_4(\mathbb{R})$ be a classifying map of the latter bundle. This means that a classifying map for the double 
is given by the following composition:
\[
\xymatrix{
G: S^2 \ar[r]^- \Delta & S^2 \times S^2 \ar[r]^-{g \times g} & BGL_4(\mathbb{R}) \times BGL_4(\mathbb{R}) \ar[r]^-j & BGL_8(\mathbb{R}) 
}
\]
Here again, by definition of the Whitney sum, the map $j$ is induced by the diagonal embedding of the two blocks, 
i.e. by the inclusion $C_{GL_8(\mathbb{R})} \subset GL_8(\mathbb{R})$! 
The equivariant classifying map $\classB$ has a canonical forgetful map to $BGL_8(\mathbb{R})$, and Lück-Uribe's argument 
explains that the composition
\[
BC_{GL_8(\mathbb{R})} \rightarrow  \classB \rightarrow BGL_8(\mathbb{R})
\]
is given by the same map $j$, induced by the canonical inclusion. This means that a classifying map for our restricted 
bundle is given precisely by the composition:
\[
\xymatrix{
G\vert_{S^1}: S^1 \ar[r] &  S^2 \ar[r]^- \Delta & S^2 \times S^2 \ar[r]^-{g \times g} & BGL_4(\mathbb{R}) \times BGL_4(\mathbb{R})
}
\]
As the first map is null-homotopic, so is the composite and we can extend $G\vert_{S^1}$ to a map on the disc.
\end{proof}

Let us think about the added disk $D^2$ above as an equatorial disk in $B^\rho$, fixed by the $\C_2$-action. 
To continue our argument, we choose as an extension  of the map 
$G\vert_{S^1}$ to $D^2$ the one provided by the map $G$ itself on the southern hemisphere of its source $S^2$. The following lemma
formalizes the fact that we can fill the northern hemisphere with the non-equivariant trivialization obtained by Floyd and extend to the southern
half-ball by applying the $\C_2$-action.

\begin{lemma}\label{lem:extendeallover}
The classifying map $f$ extends to a map $F\colon B^\rho \rightarrow \classB$.
\end{lemma}

\begin{proof}
Observe that $S^\rho \cup D^2$ is equivariantly homotopic to $S^2 \vee S^2$ with the flip action, and this a free pointed $\C_2$-space.  
By the free $\dashv$ forgetful adjunction:
\[
[S^2\vee S^2,\classB]^{\C_2} \cong [S^2, \classB^{e}] \cong [S^2, BGL_8(\mathbb{R})]
\]
where the second isomorphism comes again from \cite[Theorem 13.1]{MR3331607}. The second homotopy group of $BGL_8(\mathbb R)$
is isomorphic to $\pi_1 O(8) \cong \mathbb Z/2$.
%
Through this isomorphism $f'$ corresponds to its restriction to the union of the northern hemisphere of $S^\rho$ and the equatorial disc.
However, by construction, a classifying map for this bundle is given by $G$: we had it on the northern hemisphere and chose the southern 
part on the added disc $D^2$. In particular, this is a classifying map of twice a bundle. 
We use now Floyd's original argument to see that this bundle is trivial (it is twice something in $\mathbb Z/2$).
Therefore $f'$ is null-homotopic and thus extends to $B^\rho$.
\end{proof}

%



As an immediate consequence we have our desired result that will be used to do the surgery step:

\begin{theorem}\label{thm:equivbndletrivial}
The equivariant normal bundle to $S^\rho \hookrightarrow \widetilde{F_{10}}$ is  equivariantly trivial, i.e. the total space of the fibration is equivariantly homeomorphic to $S^\rho \times \mathbb{R}^{4\rho}$. \hfill{\qed}
\end{theorem}

\section{The Floyd manifold $F_{10}$ is a conjugation manifold}\label{sec:conjugation}

We are now ready to perform Floyd's construction of the manifolds $F_5$ and $F_{10}$ simultaneously and equivariantly
on the larger one.
By Theorem~\ref{thm:equivbndletrivial}, we learn that the equivariant tubular neighborhood of $S^\rho$ is equivariantly 
homeomorphic to  $S^\rho \times D(4\rho)$ with boundary $S^\rho \times S^{4\rho -1}$. We now remove the interior and glue back a copy of 
$ D(\rho+1 ) \times S^{4 \rho -1}$, thus killing the sphere $S^\rho$ to get $F_{10}$ with a $\C_2$-action. 
Observe that on the fixed point subset, we replaced $S^1 \times D(4)$, a tubular neighbourhood of $S^1 \hookrightarrow \widetilde{F_5}$, 
by $D(2) \times S^3$, i.e. we performed a standard non-equivariant surgery and get by Floyd's argument his manifold $F_5$. 
In particular, by construction, we have:
\[
(F_{10})^{\C_2} = F_5
\]
To prove our main result is now, we will verify the homotopical criterion for being a conjugation space from \cite{PRS}.
Let $\HF$ denote the genuine $\C_2$-equivariant spectrum associated to the constant functor with values $\F_2$,  the field of two elements.
	
\begin{theorem}
\label{thm:PRS} \cite[Theorem~1.3]{PRS}
A compact equivariant space $X$ is a conjugation space if and only if there exist finitely many non-negative integers $(m_i)_{i \in I}$ such that 
\[
X \wedge \HF \simeq \bigvee_{i \in I} S^{n_i\rho} \wedge \HF
\] 
\end{theorem}

The starting point to prove this is the very nice structural theorem of  C.~May~\cite{Clover}. Recall that $\sigma$ denotes
the sign representation.

\begin{theorem}
\label{thm:Clover} \cite[Theorem~6.13]{Clover}
Let $X$ be any pointed compact equivariant space. Denote by $S^k_{a+}$  the $k$-dimensional sphere with the antipodal action and a disjoint base-point. Then there exist two families of  finally many non-negative integers integers $(m_i,n_i)_{i \in I}\, (s_j)_{j \in J}$ and an equivariant equivalence:
\[
X \wedge \HF = \bigvee_{i \in I} S^{m_i + n_i\sigma} \wedge \HF\vee \bigvee_{j \in J} S^{s_j}_{a+} \wedge \HF.
\]
\end{theorem}

%
%
%
%

 To apply May's result, we essentially have to perform a computation in equivariant homology. For this we will analyze the two homotopy push-outs that summarize the surgery step. Let $M$ denote the complement of an equivariant tubular neighborhood of the split representation sphere
 $S^\rho \subset \widetilde{F_{10}}$. Then we have two homotopy push-out squares, that we complete by adding the cofibers (where the subscripts
only indicate the horizontal or vertical direction):

\[
\xymatrix{
S^\rho \times S^{4\rho -1} \ar@{^{(}->}[r] \ar@{^{(}->}[d] &  M \ar@{^{(}->}[d] \ar[r] & Cof_{h} \ar@{=}[d] \\ 
S^\rho  \times D(4\rho) \ar@{^{(}->}[r] \ar[d] & {\widetilde{F}_{10}} \ar[r] \ar[d] & Cof_{h} \\
Cof_{vA} \ar@{=}[r] & Cof_{vA}  & \\
 S^\rho \times S^{4\rho -1} \ar@{^{(}->}[r] \ar@{^{(}->}[d] &  M \ar@{^{(}->}[d]  \ar[r]   & Cof_{h} \ar@{=}[d]\\ 
   D(\rho +1)  \times S^{4\rho{-1}} \ar@{^{(}->}[r] \ar[d]  & F_{10} \ar[r] \ar[d] & Cof_{h}  \\
Cof_{vB} \ar@{=}[r] & Cof_{vB}  &
}
\]

Notice that the horizontal cofibers in both diagrams are the same since they share the same top horizontal cofibration. 

We will now analyze the long exact sequences for both squares, in both equivariant homotopy and non-equivariant homotopy. 
To simplify the notation, since all spaces will have free homology, we will denote the mod $2$ homology of a space that has non trivial 
classes say in degrees $n_1,n_2, n_3$  by $\HH[n_1,n_2,n_3]$.  Similarly the equivariant homology of a space $X$ such that 
$\HF \wedge X = \HF \wedge S^{V_1} \vee \HF \wedge S^{V_2} \vee\HF \wedge S^{V_3} $, will be denoted by $\HF[V_1,V_2,V_3]$.

There are three key tools to carry out the computations. Firstly,  maps out of free $\HF$-modules, like $\HF \wedge S^{n\sigma + m}$ are 
determined by the value on one generator, secondly for  free modules of the form $\HF \wedge S^{n\sigma + m}$ 
we have a preferred generator: the only class of minimal degree  that survives to a  non-trivial class when forgetting the actions, and thirdly  we know the answer in the non-equivariant setting, as these elementary computations have been carried out by Floyd~\cite{MR0334256}.

To understand the vertical cofibers in both diagrams, let us recall:

\begin{definition}\label{def:halfsmash}
The \emph{half-smash product} $A \ltimes B$ of two pointed spaces $A$ and $B$ is the cofiber of the inclusion $A \times * \hookrightarrow A \times B$.
\end{definition}

When $A$ and $B$ are $\C_2$-equivariant spaces and the base point in $B$ is a fixed point, the half-smash inherits an action of $\C_2$.

\begin{lemma}\label{lem:halfsmash}
The homotopy cofiber of the projection $A\times B \rightarrow A$ is homotopy equivalent to $\Sigma (A \ltimes B) \simeq \Sigma (A \wedge B) \vee \Sigma B$
and the map from $A$ is null-homotopic.
\end{lemma}
\begin{proof}
Factor the identity as $A \hookrightarrow A \times B \rightarrow A$. The associated cofibration sequence 
$A \ltimes B \rightarrow Cof(Id_A) \rightarrow Cof(A\times B \rightarrow A)$ yields the desired result. The decomposition as
a wedge is a consequence of the splitting of the suspension of any product and that the map $A \rightarrow \Sigma (A \ltimes B)$
is null-homotopic follows from the fact that the projection map has a retraction.
\end{proof}

This applies again to the equivariant setting when the base point of the second factor is a fixed point (and the maps we consider are thus equivariant). 
Our interest lies in the half-smash product of representation spheres.

\begin{corollary}\label{cor:halfsmash}
Let $V_1$ and $V_2$ be two orthogonal representations of $\C_2$. The homotopy cofiber of the projection $S^{V_1} \times S^{V_2} \rightarrow S^{V_1}$
is then equivalent to  $S^{V_1+ V_2+1} \vee S^{V_2+1}$.
\end{corollary}

In particular we have a stable equivariant splitting
\[
\HF \wedge  (S^{V_1} \ltimes S^{V_2}) \simeq \HF \wedge S^{V_1+1} \vee \HF \wedge S^{V_1 + V_2 +1} = \HF[V_1 +1, V_1+V_2+1]
\]




Smashing the first homotopy pushout with the spectrum $\HF$, we get a homotopy pushout of equivariant spectra:
\[
\xymatrix{
\HF[\rho,4\rho-1,5\rho-1] \ar[r] \ar[d] & \HF \wedge M \ar[d] \ar[r] & \HF \wedge Cof_h \ar@{=}[d] \\
\HF[\rho] \ar[d]^0 \ar[r]^-{\textcircled{\tiny{1}}} & \HF[\rho,2\rho,3\rho,4\rho,5\rho] \ar[d] \ar[r]^-{\textcircled{\tiny{2}}} & \HF \wedge Cof_h \\
\HF[4\rho,5\rho] \ar@{=}[r] & \HF[4\rho,5\rho] & 
}  
\]
where we have applied Corollary~\ref{cor:halfsmash} and May's Theorem~\ref{thm:Clover} to identify the spectra. The map labelled with a zero is
null-homotopic, see Lemma~\ref{lem:halfsmash}. We proceed to identify the two arrows labeled $\textcircled{\tiny{1}}$ and $\textcircled{\tiny{2}}$.

\begin{lemma}\label{lem:HFofM}
The map $\HF \wedge M \rightarrow \HF \wedge Cof_h$ is the unique map of $\HF$-modules
$\HF[\rho,2\rho,3\rho] \rightarrow \HF[2\rho,3\rho,4\rho,5\rho]$
that collapses the $\HF[\rho]$ factor and identifies the $\HF[2\rho,3\rho]$ factor with the corresponding part in the target.
\end{lemma}
\begin{proof}

Because the fibration $\widetilde{F_{10}} \rightarrow S^\rho$ splits, and the arrow $\textcircled{\tiny{1}}$ is induced by an arbitrary splitting 
(they are all homotopic, since the splittings we consider come from the choice of a fixed point in the sphere $S^{4\rho-1}$, and there the 
fixed points form a connected space), we get that $\textcircled{\tiny{1}}$ is a split monomorphism, hence from the triangulated structure 
for $\HF$-modules in the equivariant stable homotopy category,  $\textcircled{\tiny{2}}$ is a split epimorphism. 

We can now construct a commutative  solid  diagram:
\[
\xymatrix{
\HF[\rho] \ar[r] \ar@{=}[d]& \HF[\rho]\oplus \HF[2\rho,3\rho,4\rho,5\rho] \ar[r] \ar[d]^{\textcircled{\tiny{1}}\oplus Can} & HF[2\rho,3\rho,4\rho,5\rho] \ar@{..>}[d] \\
\HF[\rho] \ar[r] &  \HF[\rho, 2\rho,3\rho,4\rho,5\rho] \ar[r] & \HF\wedge Cof_h
}
\]
By the triangulated structure,  the dotted arrow exists and enhances the diagram to a morphism of triangles. 
By construction the middle arrow is an epimorphism in homotopy and a $\HF_\star$-module map. Because this ring is local, 
by Nakayama's Lemma, the map is an isomorphism. The five lemma shows then that the dotted arrow is an isomorphism, and 
by identifying $\HF \wedge Cof_h$ via this isomorphism we have as desired that $\HF \wedge Cof_h \simeq \HF[2\rho,3\rho,4\rho,5\rho]$.
 Likewise the vertical map $\HF[\rho,2\rho,3\rho,4\rho,5\rho] \rightarrow \HF[4\rho,5\rho]$ is the projection, which
allows us to identify $\HF \wedge M$ with $\HF[\rho,2\rho,3\rho]$ and the map $\HF \wedge M \rightarrow \HF \wedge Cof_h$.
\end{proof}

We finally come to our main result. We use now our second pushout diagram in which we plug in the previous identification.

\begin{theorem}\label{thmF10isconjugation}
The manifold $F_{10}$ is a smooth conjugation manifold with fixed points $F_5$.
\end{theorem}

\begin{proof}
We use Lemma~\ref{lem:HFofM} and apply the same strategy to the second pushout diagram smashed with $\HF$; we however  rotate the horizontal triangles to the left to simplify the final computation. The identification of the vertical cofiber comes again from Corollary~\ref{cor:halfsmash}.
\[
\xymatrix{
\HF[2\rho-1,3\rho-1,4\rho-1,5\rho-1] \ar@{>}[r]^-{\textcircled{\tiny{1}}} \ar@{=}[d] & \HF[\rho,4\rho-1,5\rho-1] \ar[r] \ar[d]^-{\textcircled{\tiny{2}}}  
& \HF[\rho,2\rho,3\rho] \ar[d] \\ 
\HF[2\rho-1,3\rho-1,4\rho-1,5\rho-1] \ar@{>}[r]^-{\textcircled{\tiny{3}}}  & \HF[4\rho-1] \ar[d]^0 \ar[r] & \HF \wedge F_{10} \ar[d] \\
&  \HF[\rho+1,5\rho] \ar@{=}[r] & \HF[\rho+1,5\rho] 
}  
\]
We know from Lemma~\ref{lem:HFofM} that the map ${\textcircled{\tiny{1}}}$ maps $4\rho-1 \mapsto 4\rho-1$ and that 
${\textcircled{\tiny{2}}}$ is the canonical projection, hence ${\textcircled{\tiny{3}}}$ is the split epimorphism projection onto the $4\rho-1$
component. Therefore
$\HF\wedge F_{10} \simeq \HF[2\rho,3\rho,5\rho]$ and we conclude by Theorem~\ref{thm:PRS} that $F_{10}$
is a conjugation space. The fixed points have already been identified earlier as $F_5$ and we have constructed the larger
manifold as a smooth $\C_2$-equivariant manifold.
\end{proof}

\bibliographystyle{plain}\label{biblography}

\end{document}